\documentclass[11pt]{amsart}

\usepackage[margin=1in]{geometry}
\usepackage{amssymb}
\usepackage[hidelinks]{hyperref}
\usepackage[capitalize]{cleveref}
\usepackage{microtype}

\usepackage{tikz}

\newcommand{\E}{\mathbf{E}}

\newtheorem{theorem}{Theorem}

\crefname{theorem}{Theorem}{Theorems}

\newtheorem{thm}[theorem]{Theorem}

\newtheorem{lem}[theorem]{Lemma}
\newtheorem{prop}[theorem]{Proposition}

\theoremstyle{definition}

\title{Embedding distance graphs in finite field vector spaces}

\author{Alex Iosevich \and Hans Parshall}

\begin{document}

\begin{abstract}
We show that large subsets of vector spaces over finite fields determine certain point configurations with prescribed distance structure.  More specifically, we consider the complete graph with vertices as the points of $A \subseteq \mathbf{F}_q^d$ and edges assigned the algebraic distance between pairs of vertices.  We prove nontrivial results on locating specified subgraphs of maximum vertex degree at most $t$ in dimensions $d \geq 2t$. 
\end{abstract}

\maketitle

\section{Introduction}

The famous Erd\H{o}s and Falconer distance problems aim to quantify the extent to which large sets must determine many distances.  A near optimal result on the planar Erd\H{o}s distinct distance problem was established by Guth and Katz \cite{guth15}, who showed that every set of $N$ points in $\mathbf{R}^2$ determines at least a constant multiple of $\frac{N}{\log(N)}$ distinct distances.  However, even in the Euclidean plane, the Falconer distance problem continues to present a gap in our understanding of distance sets. Falconer~\cite{falconer85} proved that when the Hausdorff dimension of compact $A \subseteq \mathbf{R}^d$ is large, specifically $\dim(A) >\frac{d+1}{2}$, the distance set $\{|x - y| : x,y \in A\} \subseteq \mathbf{R}$ must have positive Lebesgue measure.  The state of the art by Wolff~\cite{wolff99} in dimension $d = 2$ and Erdo\u{g}an~\cite{erdogan05} in dimensions $d \geq 3$ establishes the same conclusion with the weaker hypothesis of $\dim(A) > \frac{d}{2}+\frac{1}{3}$.  Falconer conjectures that this can be relaxed further to $\dim(A)>\frac{d}{2}$, which would be optimal. Significant progress over these exponents has been achieved recently by Orponen~\cite{orponen17}, Shmerkin~\cite{shmerkin17} and Keleti and Shmerkin~\cite{keleti18} in the setting of sets that are close to being Ahlfors-David regular. 

More generally, many authors have considered the related problem of locating point configurations with prescribed metric structure within subsets $A \subseteq \mathbf{R}^d$ of large Hausdorff dimension.  For instance, one result in this direction of Bennett, the first listed author and Taylor~\cite[Theorem 1.7]{bennett16-chains} implies:

\begin{theorem}[\cite{bennett16-chains}]\label{chainsTheorem} For $d \geq 2$, let $A \subseteq \mathbf{R}^d$ be compact with $\dim(A)>\frac{d+1}{2}$.  For any $k \in \mathbf{N}$, there exists an open interval $(a,b) \subset \mathbf{R}$ where for every $\lambda \in (a,b)$, there exist distinct $x_0, \ldots, x_k \in A$ with $|x_{j + 1} - x_j| = \lambda$ for $0 \leq j < k$.
\end{theorem}
One can interpret this statement in terms of equilateral Euclidean distance graphs $\mathcal{G}_\lambda(A)$, defined for $A \subseteq \mathbf{R}^d$ and $\lambda \in \mathbf{R}$, where we consider the points of $A$ as vertices and connect every pair of points $x,y \in A$ with an edge exactly when $|x - y| = \lambda$.  With this notion, \cref{chainsTheorem} states that for every compact $A \subseteq \mathbf{R}^d$ with $\dim(A)>\frac{d+1}{2}$ and for arbitrarily large $k \in \mathbf{N}$, there exists an open interval of distances $\lambda \in \mathbf{R}$ for which $\mathcal{G}_\lambda(A)$ contains a path of length $k$.  Under the stronger hypothesis $\dim(A)>\frac{d+3}{2}$, Greenleaf, the first listed author and Pramanik~\cite{greenleaf17} demonstrate an open interval of distances $\lambda \in \mathbf{R}$ for which $\mathcal{G}_\lambda(A)$ contains cycles of length $2k$.  Positive results for cycles of length 3 within $\mathcal{G}_\lambda(A)$ in $\mathbf{R}^4$, corresponding to the vertices of equilateral triangles, were recently obtained~\cite{iosevich16}, but there seems to be a gap in the literature for the general case of arbitrarily long odd length cycles with equilateral edge lengths.  Additional results on locating more general distance graph structure within sets of large Hausdorff dimension have been obtained by Chatzikonstantinou, the first listed author, Mkrtchyan and Pakianathan~\cite{chatz17}.

Our goal here is to illustrate a related approach for locating distance graph structure within large subsets of $\mathbf{F}_q^d$, the vector space of dimension $d$ over a finite field with odd characteristic.  For $v,w \in \mathbf{F}_q^d$, we consider their usual dot product 
\[
	v \cdot w := \sum_{j = 1}^d v_jw_j
\]
and refer to $|v|^2 := v \cdot v$ as the \emph{length} of $v$ and $|v - w|^2$ as the \emph{distance} between $v$ and $w$.  While this notion of distance is certainly not a metric, this setup provides a useful model setting for the Falconer distance problem and its variants.  For instance, the first listed author and Rudnev~\cite{iosevich07} demonstrated that every $A \subseteq \mathbf{F}_q^d$ with $|A| > 2q^{\frac{d + 1}{2}}$ determines all distances, in the sense that $\{|x - y|^2 : x,y \in A\} = \mathbf{F}_q$.  An analogue to Wolff's result in this setting was first achieved by Chapman, Erdo\u{g}an, Hart, the first listed author and Koh~\cite{chapman12}, who showed that when $q \equiv 3 \mod{4}$, every $A \subseteq \mathbf{F}_q^2$ with $|A| \geq q^{4/3}$  determines at least a constant proportion of $q$ distances.  This was shown to further hold when $q \equiv 1 \mod{4}$ by Bennett, Hart, the first listed author, Pakianathan and Rudnev~\cite{bennett17} through a group action method.  Analogous problems for locating point configurations with prescribed distance structure have been studied by several authors.  Many of these results show that large sets determine isometric copies of a positive proportion of a class of point configurations (for instance, \cite{bennett17, chapman12}), but here we will show that large sets determine isometric copies of \emph{all} point configurations of a fixed class (as was done in \cite{parshall17, vinh12}).

We find it convenient to organize distance structure within subsets $A \subseteq \mathbf{F}_q^d$ in the language of graph theory.  We call a graph $\mathcal{G} = (V,E)$ a \emph{distance graph} when to each edge $e \in E$ there is some associated nonzero length $\lambda_{e} \in \mathbf{F}_q^*$.  We call $X \subseteq \mathbf{F}_q^d$ an \emph{isometric copy} of $\mathcal{G}$ when there exists a distance preserving bijection $\varphi : V \rightarrow X$ where for every $v,w \in V$ connected by an edge $e \in E$, $|\varphi(v) - \varphi(w)|^2 = \lambda_e$.  Our main result is the following.

\begin{thm}\label{introMain} Let $n, t \in \mathbf{N}$, and let $A \subseteq \mathbf{F}_q^d$ with $|A| \geq 12n^2 q^{\frac{d - 1}{2} + t}$.  Then $A$ contains an isometric copy of every distance graph with $n$ vertices and maximum vertex degree $t$. \end{thm}

The main analytic input into Theorem \ref{introMain} is Theorem \ref{distancesTheorem} below. An interested reader can check that the ``distance'' function ${|\cdot|}^2$ can be replaced by any non-degenerate quadratic form using the results of \cite{iosevich07} and \cite{bennett17}. We stick to the ``Euclidean distance'' in this context for the sake of notational clarity. 

In particular, this allows us to locate arbitrarily large configurations of maximum degree $t$ within large subsets of $\mathbf{F}_q^{2t}$ provided $q$ is taken sufficiently large with respect to the desired number of vertices $n$.  For instance, we recover nontrivial results for all cycle graphs of length $n$ in dimensions $d \geq 4$, which is new.  Previous results of this type typically require the dimension $d$ of $\mathbf{F}_q^d$ to depend on the number of sought after vertices.  Vinh~\cite{vinh12} has proven that for a constant $C$ depending only upon $n$, if $A \subseteq \mathbf{F}_q^d$ with $|A| > C q^{\frac{d - 3}{2} + n}$, then $A$ contains an isometric copy of every complete graph on $n$ vertices.  While our methods differ, \cref{introMain} recovers Vinh's exponent of $\frac{d - 3}{2} + n$ for complete graphs on $n$ vertices.  Vinh's result yields a nontrivial exponent for every distance graph by considering an embedding into a complete graph.  \cref{introMain} yields an improvement over this exponent whenever the maximum degree $t$ is strictly less than $n - 1$.

There are two cases where \cref{introMain} does not even meet the quantitative information provided by previous similar results.  Bennett~et~al.~\cite{bennett16-paths} showed that if $A \subseteq \mathbf{F}_q^d$ with $|A| > 4(n + 1) q^{\frac{d + 1}{2}}$, then $A$ contains an isometric copy of every path with $n$ vertices.  This is nontrivial in dimensions $d \geq 2$, while applying \cref{introMain} for paths is unfortunately only nontrivial for $d \geq 4$.  A more careful application of our methods yields a result for paths in dimensions $d \geq 3$, but we do not pursue this.  Additionally, the second listed author~\cite{parshall17} has proven results for complete graphs on $n$ vertices that are nontrivial in dimensions $d \geq n$ provided one is willing to impose technical conditions on the edge distances; we instead opt for a result uniform in nonzero edge distances.

Our strategy is to proceed by an edge deletion induction.  This relies on a functional distance counting theorem, \cref{distancesTheorem}, which previously appeared in the work of Bennett~et~al.~\cite{bennett16-paths}.  As this is our main tool, we record its proof after setting notation and recalling some facts about exponential sums over finite fields.  Before proving a more quantitative form of \cref{introMain} in the final section, we demonstrate our method with the small case of equilateral triangles; this helps to highlight why our methods are constrained by the maximum vertex degree.

%
%

\section{Preliminaries}

Let $\chi : \mathbf{F}_q \rightarrow \mathbf{C}$ denote the canonical additive character.  More precisely, if $p$ is the characteristic of $\mathbf{F}_q$ and $\operatorname{Tr} : \mathbf{F}_q \rightarrow \mathbf{F}_p$ is the usual trace function, we define for $a \in \mathbf{F}_q$
\[
	\chi(a) := \exp\Big(\frac{2\pi i \operatorname{Tr}(a)}{p}\Big).
\]
We find it convenient to use averaging notation for $f : \mathbf{F}_q^d \rightarrow \mathbf{C}$ of
\[
	\E_x f(x) := q^{-d} \sum_{x \in \mathbf{F}_q^d} f(x),
\]
and we will condense multiple averages $\E_{x_1} \E_{x_2} \cdots \E_{x_n}$ as $\E_{x_1, x_2, \ldots, x_n}$.  With this notation, the familiar orthogonality relation for $\chi$ and $y \in \mathbf{F}_q^d$ takes the form
\[
	\E_x \chi(x \cdot y) = \begin{cases} 1 & \text{if } y = 0 \\ 0 & \text{if } y \in \mathbf{F}_q^d \setminus \{0\}. \end{cases}
\]
For $\lambda \in \mathbf{F}_q$, $x \in \mathbf{F}_q^d$ we define the normalized indicator function
\[
	\sigma_\lambda(x) := \begin{cases} q & \text{if } |x|^2 = \lambda \\ 0 & \text{otherwise.} \end{cases}
\]
This plays the role of the surface measure of the sphere of radius $\lambda$.  For convenience we perform the standard calculation that shows $\sigma_\lambda$ is essentially $L^1$-normalized, with error tending to zero as $q \rightarrow \infty$ provided $d \geq 2$.
\begin{lem}\label{sigmaNormalizedLemma}
	For any $\lambda \in \mathbf{F}_q^*$,
	\begin{equation}\label{sigmaNormalized}
	|\E_x \sigma_\lambda(x) - 1| \leq q^{-\frac{d-1}{2}}.
\end{equation}	
\end{lem}
\begin{proof}
	By orthogonality, we have
	\begin{align*}
		\E_x \sigma_\lambda(x) &= \E_x \sum_{\ell \in \mathbf{F}_q} \chi(\ell |x|^2 - \ell \lambda)\\
		&= \sum_{\ell \in \mathbf{F}_q} \chi(-\ell \lambda) \E_x \chi(\ell |x|^2)\\
		&= 1 + \sum_{\ell \in \mathbf{F}_q^*} \chi(-\ell \lambda) \E_x \chi(\ell |x|^2).
	\end{align*}
	Introducing the quadratic character $\eta : \mathbf{F}_q^* \rightarrow \{-1,1\}$ defined for $a \in \mathbf{F}_q^*$ by
	\[
		\eta(a) := \begin{cases} 1 & \text{if } a \text{ is a square in } \mathbf{F}_q^* \\ -1 & \text{otherwise,} \end{cases}
	\]
	and the Gaussian sum
	\[
		G(\chi,\eta) := \sum_{a \in \mathbf{F}_q^*} \chi(a)\eta(a),
	\]
	it follows from \cite[Theorem 5.33]{lidl97} that we have
	\[
		\E_x \chi(\ell |x|^2) = q^{-d} G(\chi,\eta)^d \eta(\ell)^d .
	\]
	Together with our computation above,
	\[
		|\E_x \sigma_\lambda(x) - 1| \leq q^{-d} |G(\chi,\eta)|^d \Big|\sum_{\ell \in \mathbf{F}_q^*} \chi(-\ell \lambda) \eta(\ell)^d\Big|.
	\]
	As $\lambda \in \mathbf{F}_q^*$, we can reindex in $\ell$ for
	\[
		|\E_x \sigma_\lambda(x) - 1| \leq q^{-d} |G(\chi,\eta)|^d \Big|\sum_{\ell \in \mathbf{F}_q^*} \chi(\ell) \eta(\ell)^d\Big|
	\]
	It is not hard \cite[Theorem 5.15]{lidl97} to show $|G(\chi,\eta)| = q^{1/2}$.  When $d$ is even, the sum in $\ell$ amounts to $-1$ (by orthogonality of $\chi$), and when $d$ is odd, the sum in $\ell$ contributes another Gaussian sum.  That is, we could slightly improve \eqref{sigmaNormalized} in even dimensions; in any case we can conclude \eqref{sigmaNormalized}.
\end{proof}

\section{A Functional Distance Theorem}

The first listed author and Rudnev~\cite{iosevich07} showed not only that large sets determine all distances, but that distances within large subsets equidistribute.  We will use a functional version of their result, which appears in the work of Bennett et al.~\cite[Theorem 2.1]{bennett16-paths}.  More precisely, with the $L^2$ norm of $f : \mathbf{F}_q^d \rightarrow \mathbf{C}$ normalized as
\[
	\| f \|_2 := \Big(\E_x |f(x)|^2 \Big)^{1/2},
\]
our main distance counting tool is the following.
\begin{thm}[\cite{bennett16-paths}]\label{distancesTheorem} For any $f,g : \mathbf{F}_q^d \rightarrow \mathbf{C}$ and $\lambda \in \mathbf{F}_q^*$,
\[
	\Big|\E_{x,y} f(x)g(y)\sigma_\lambda(x - y) - \E_x f(x) \E_y g(y)\Big| \leq 2q^{-\frac{d-1}{2}} \| f \|_2 \| g \|_2.
\]
\end{thm}

In order to include the proof of \cref{distancesTheorem}, we recall some Fourier analytic facts here.  The \emph{Fourier transform} $\widehat{f} : \mathbf{F}_q^d \rightarrow \mathbf{C}$ for $f : \mathbf{F}_q^d \rightarrow \mathbf{C}$ is defined for $\xi \in \mathbf{F}_q^d$ by
\[
	\widehat{f}(\xi) = \E_x f(x) \chi(x \cdot \xi).
\]
The orthogonality of $\chi$ allows one to quickly conclude the Fourier inversion formula
\[
	f(x) = \sum_{\xi \in \mathbf{F}_q^d} \widehat{f}(\xi) \chi(-x \cdot \xi)
\]
and the Plancherel identity
\[
	\E_x |f(x)|^2 = \sum_{\xi \in \mathbf{F}_q^d} |\widehat{f}(\xi)|^2.
\]
While \cref{sigmaNormalizedLemma} relied only on Gaussian sums, \cref{distancesTheorem} relies on more delicate cancellation.  To be precise, for $a,b \in \mathbf{F}_q$ let us define the \emph{Kloosterman sum}
\begin{align*}
	K(a,b) := \sum_{\ell \in \mathbf{F}_q^*} \chi\Big(as + \frac{b}{s}\Big)
\end{align*}
and the \emph{Sali\'{e} sum}
\begin{align*}
	S(a,b) := \sum_{\ell \in \mathbf{F}_q^*} \chi\Big(as + \frac{b}{s}\Big) \eta(s).
\end{align*}
If either $a$ or $b$ is nonzero, both bounds $|S(a,b)| \leq 2\sqrt{q}$ and $|K(a,b)| \leq 2\sqrt{q}$ are well-known; see, for instance, \cite{iwaniec04}.  These bounds are somewhat more involved than showing $|G(\chi,\eta)| = q^{1/2}$.  While the Sali\'{e} sum can indeed be related back to a Gaussian sum, the Kloosterman sum bound relies on the fundamental work of Weil~\cite{weil48}.  With these bounds, we are ready to prove the functional distance theorem.

\begin{proof}[Proof of \cref{distancesTheorem}]
	Proceeding as in \cref{sigmaNormalizedLemma}, we apply orthogonality for
	\begin{align*}
		\E_{x,y}f(x)g(x + y)\sigma_\lambda(y) &= \E_{x,y} f(x)g(x + y) \sum_{\ell \in \mathbf{F}_q} \chi(\ell |y|^2)\\
		&= \E_x f(x) \E_y g(y) + \sum_{\ell \in \mathbf{F}_q^*} \chi(-\ell \lambda) \E_{y} \chi(\ell |y|^2) \E_x f(x)g(x + y)
	\end{align*}
	Applying Fourier inversion,
	\[
		\E_x f(x)g(x + y) = \sum_{\xi \in \mathbf{F}_q^d} \widehat{f}(\xi)\widehat{g}(-\xi)\chi(y \cdot \xi),
	\]
	so we have shown
	\[
	\E_{x,y}f(x)g(x + y)\sigma_\lambda(y) - \E_x f(x) \E_y g(y) = \sum_{\xi \in \mathbf{F}_q^d} \widehat{f}(\xi)\widehat{g}(-\xi) \sum_{\ell \in \mathbf{F}_q^*} \chi(-\ell \lambda) \E_{y} \chi(\ell |y|^2 + y \cdot \xi)
	\]
	By \cite[Theorem 5.33]{lidl97},
	\[
		\E_{y} \chi(\ell |y|^2 + y \cdot \xi) = q^{-d} G(\chi,\eta)^d \chi\Big(-\frac{|\xi|^2}{4\ell}\Big) \eta(\ell)^d
	\]
	which we combine with our previous work to obtain
	\[
		\Big|\E_{x,y}f(x)g(x + y)\sigma_\lambda(y) - \E_x f(x) \E_y g(y)\Big| \leq q^{-d/2} \sum_{\xi \in \mathbf{F}_q^d} |\widehat{f}(\xi)||\widehat{g}(-\xi)| \Big|\sum_{\ell \in \mathbf{F}_q^*} \chi\Big(-\ell \lambda - \frac{|\xi|^2}{4\ell}\Big) \eta(\ell)^d\Big|.
	\]
	As $\lambda \neq 0$, we can recognize the sum in $\ell$ as either a Kloosterman sum (when $d$ is even) or a Sali\'{e} sum (when $d$ is odd) and arrive at
	\[
	\Big|\E_{x,y}f(x)g(x + y)\sigma_\lambda(y) - \E_x f(x) \E_y g(y)\Big| \leq 2q^{-\frac{d - 1}{2}} \sum_{\xi \in \mathbf{F}_q^d} |\widehat{f}(\xi)||\widehat{g}(-\xi)|.
	\]
	Applying Cauchy-Schwarz and Plancherel,
	\[
		\sum_{\xi \in \mathbf{F}_q^d} |\widehat{f}(\xi)||\widehat{g}(-\xi)| \leq \| f \|_2 \| g \|_2,
	\]
	completing the proof.
\end{proof}

\section{Equilateral Triangles}

The rough plan is to argue that a properly normalized count of isometric copies of a particular graph $\mathcal{G}$ within sufficiently large subset $A \subseteq \mathbf{F}_q^d$ is roughly equal to a properly normalized count of isometric copies of the subgraph of $\mathcal{G}$ formed by deleting an edge.  We will illustrate our approach through the small example of counting equilateral triangles, by which we mean complete graphs on 3 vertices that are connected by edges of the same length.  This will amount to a total of three edge deletions, see the figure, and each edge deletion requires an application of \cref{distancesTheorem}, resulting in an error term.  Provided we can keep the sum of these error terms under control, we will find ourselves counting points in $A$ with no distance restrictions at all.

\begin{prop}\label{triangles} Let $\lambda \in \mathbf{F}_q^*$, $A \subseteq \mathbf{F}_q^d$ with $|A| = \alpha q^d$ and $\alpha \geq 24q^{2 - \frac{d + 1}{2}}.$ Then
\begin{equation}\label{trianglesEquation}
	\E_{x,y,z} 1_A(x)1_A(y)1_A(z) \sigma_{\lambda}(x - y) \sigma_{\lambda}(x - z) \sigma_{\lambda}(y - z) \geq \frac{1}{2}\alpha^3.
\end{equation}
\end{prop}

\begin{proof}
	Set $f_x(y) = 1_A(y) \sigma_{\lambda}(x - y)$.  Applying \cref{distancesTheorem} for a fixed $x$,
	\[
		|\E_{y,z} f_x(y)f_x(z) \sigma_{\lambda}(y - z) - \E_y f_x(y) \E_z f_x(z)| \leq 2q^{\frac{1 - d}{2}} \| f_x \|_2^2
	\]
	Note that $f_x$ is not $L^2$-normalized, so we pick up an extra factor of $q$.  That is,
	\begin{align*}
		\| f_x \|_2^2 &= \E_y 1_A(y)\sigma_{\lambda}(x - y)^2\\
		&= q \E_y f_x(y).
	\end{align*}
	Averaging over potential vertices $x$ in $A$ and applying the triangle inequality, we see
	\[
		|\E_{x,y,z} 1_A(x) f_x(y)f_x(z) \sigma_{\lambda}(y - z) - \E_{x,y,z} 1_A(x) f_x(y) f_x(z)| \leq 2q^{\frac{3 - d}{2}} \E_x 1_A(x) \E_y f_x(y).
	\]
	Recall $f_x(y) = 1_A(y)\sigma_\lambda(x - y)$.  Since $\alpha$ is not too small, \cref{distancesTheorem} tells us
	\[
		\E_x 1_A(x) 1_A(y) \sigma_\lambda(x - y) \leq 2\alpha^2.
	\]
	Combining this with our assumption that $\alpha$ is not too small, we have shown
	\[
		|\E_{x,y,z} 1_A(x) f_x(y)f_x(z) \sigma_{\lambda_3}(y - z) - \E_{x,y,z} 1_A(x) f_x(y) f_x(z)| \leq \frac{1}{6} \alpha^3.
	\]
What we have done so far is to show that the normalized count of equilateral triangles within $A$, $\{x,y,z\} \subseteq A$ each of distance $\lambda$ from each other, is roughly equal to the normalized count of paths of length 2 within $A$, $\{x,y,z\} \subseteq A$ where both $y$ and $z$ are distant $\lambda$ from $x$, where the error amounts to a small constant proportion of the expected count of $\alpha^3$.  In other words, by deleting an edge length from our desired configuration and counting the resulting configurations, we are still counting a positive proportion of the expected number of desired configurations.
	
	We now aim to delete another edge to estimate $\E_{x,y,z} 1_A(x) 1_A(y) \sigma_\lambda(x - y) 1_A(z) \sigma_\lambda(x - z)$.  Continuing with the notation $f_z(x) = 1_A(x) \sigma_\lambda(x - z)$, we apply \cref{distancesTheorem} for a fixed $z$ to obtain
	\[
		|\E_{x,y} f_z(x) 1_A(y) \sigma_{\lambda}(x - y) - \alpha \E_x f_z(x)| \leq 2 \alpha^{1/2} q^\frac{1 - d}{2} \| f_z \|_2.
	\]
	Averaging over potential vertices $z$ in $A$ and applying the triangle inequality,
	\[
		|\E_{x,y,z} f_z(x) 1_A(y) 1_A(z) \sigma_{\lambda}(x - y) - \alpha \E_{x,z} f_z(x) 1_A(z)| \leq 2\alpha^{1/2} q^\frac{1 - d}{2} \E_z 1_A(z) \| f_z \|_2.
	\]
	Applying Cauchy-Schwarz,
	\begin{align*}
		\E_z 1_A(z) \| f_z \|_2 &\leq \alpha^{1/2} \E_z 1_A(z) \| f_z \|_2^2\\
		&= \alpha^{1/2} q^{1/2} (\E_{x,z} 1_A(x) 1_A(z) \sigma_\lambda(x - z))^{1/2}\\
		&\leq 2\alpha^{3/2} q^{1/2},
	\end{align*}
	where again we have used that $\alpha$ is not too small to productively apply \cref{distancesTheorem}.  In total,
	\[
	|\E_{x,y,z} f_z(x) 1_A(y) 1_A(z) \sigma_{\lambda}(x - y) - \alpha \E_{x,z} f_z(x) 1_A(z)| \leq 4\alpha^2 q^\frac{2 - d}{2},
	\]
	where the error here is smaller than our previous error by a factor of $q$.  In particular,
	\[
	|\E_{x,y,z} 1_A(x)1_A(y)1_A(z)\sigma_\lambda(x - y)\sigma_\lambda(x - z)\sigma_\lambda(y - z) - \alpha \E_{x,z} 1_A(x) 1_A(z) \sigma_{\lambda}(x - z)| \leq \frac{1}{3}\alpha^3.
	\]
	We have now reduced the problem of counting equilateral triangles of side length $\lambda$ within $A$ to the problem of counting pairs of points of distance $\lambda$ apart within $A$, which is exactly what \cref{distancesTheorem} does for us.  As $\alpha$ is not too small,
	\[
		|\E_{x,z} 1_A(x) 1_A(z) \sigma_\lambda(x - z) - \alpha^2| \leq 2\alpha q^\frac{1 - d}{2}
	\]
	which together with the previous reductions leads quickly to
	\[
		|\E_{x,y,z} 1_A(x)1_A(y)1_A(z)\sigma_\lambda(x - y)\sigma_\lambda(x - z)\sigma_\lambda(y - z) - \alpha^3| \leq \frac{1}{2} \alpha^3,
	\]
	from which \eqref{trianglesEquation} follows immediately.
\end{proof}

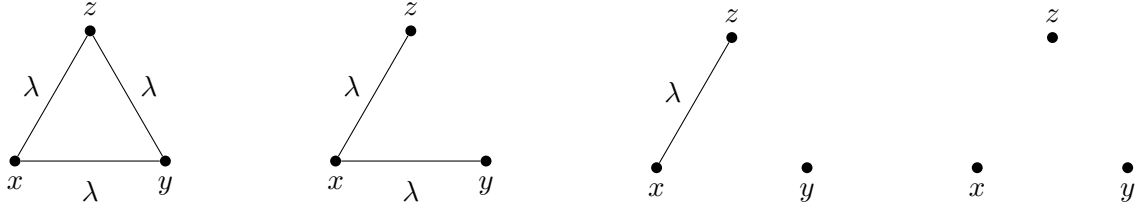
\begin{figure}

\caption{Each application of \cref{distancesTheorem} deletes an edge to count simpler configurations.}

\hspace{1em}
\begin{tikzpicture}
\node[name=X, fill=black, shape=circle, scale=.4, label=below:{$x$}] at (-1,0) {};
\node[name=Y, fill=black, shape=circle, scale=.4, label=below:{$y$}] at (1,0) {};
\node[name=Z, fill=black, shape=circle, scale=.4, label=above:{$z$}] at (0,1.732) {};

\draw (X) -- (Y);
\node[fill=none, label=below:{$\lambda$}] at (0,0) {};

\draw (X) -- (Z);
\node[fill=none, label=above left:{$\lambda$}] at (-.4,.6) {};

\draw (Y) -- (Z);
\node[fill=none, label=above right:{$\lambda$}] at (.4,.6) {};

\end{tikzpicture}
\hspace{4em}
\begin{tikzpicture}

\node[name=X, fill=black, shape=circle, scale=.4, label=below:{$x$}] at (-1,0) {};
\node[name=Y, fill=black, shape=circle, scale=.4, label=below:{$y$}] at (1,0) {};
\node[name=Z, fill=black, shape=circle, scale=.4, label=above:{$z$}] at (0,1.732) {};

\draw (X) -- (Y);
\node[fill=none, label=below:{$\lambda$}] at (0,0) {};

\draw (X) -- (Z);
\node[fill=none, label=above left:{$\lambda$}] at (-.4,.6) {};

\end{tikzpicture}
\hspace{4em}
\begin{tikzpicture}

\node[name=X, fill=black, shape=circle, scale=.4, label=below:{$x$}] at (-1,0) {};
\node[name=Y, fill=black, shape=circle, scale=.4, label=below:{$y$}] at (1,0) {};
\node[name=Z, fill=black, shape=circle, scale=.4, label=above:{$z$}] at (0,1.732) {};

\draw (X) -- (Z);
\node[fill=none, label=above left:{$\lambda$}] at (-.4,.6) {};

\end{tikzpicture}
\hspace{4em}
\begin{tikzpicture}

\node[name=X, fill=black, shape=circle, scale=.4, label=below:{$x$}] at (-1,0) {};
\node[name=Y, fill=black, shape=circle, scale=.4, label=below:{$y$}] at (1,0) {};
\node[name=Z, fill=black, shape=circle, scale=.4, label=above:{$z$}] at (0,1.732) {};

\end{tikzpicture}
\hspace{1em}

\end{figure}

In the argument above, deleting an edge via \cref{distancesTheorem} allowed us to count simpler configurations at the price of including additional error resulting from our functions $f_x$ not being $L^2$-normalized.  This is the heart of our argument, and the rest is essentially bookkeeping.  The quality of the error terms in each edge deletion above depended only upon the degrees of the vertices at the end of the chosen edge.  This should give some hope that for counting cycles of length $n$, we may have to sum more error terms (an amount depending upon $n$), but each power of $q$ appearing is at worst $q^{\frac{3 - d}{2}}$. Roughly, as we proceed through the general argument in the next section, we will see that deleting an edge between two vertices of degree at most $t$ amounts to increasing the $q^{\frac{1 - d}{2}}$ factor appearing in the error from \cref{distancesTheorem} to a factor of $q^{t + \frac{1 - d}{2}}$.  This accounts for the necessary size constraint imposed on $\alpha$.  It should be believable that the restriction to equilateral configurations, while conceptually convenient, is not particularly important since we delete one edge at a time.

It is reasonable to worry that \cref{triangles} does not guarantee us three \emph{distinct} points $\{x,y,z\} \subseteq A$ forming the desired triangle; this is required to make any conclusion about $A$ containing an isometric copy of a complete graph on 3 vertices with 3 edge lengths of $\lambda$.  It is not terribly surprising that coincidences between these vertices amount to a well-controlled error term, and we make this precise at the end of the next section in the general case.

\section{Edge Deletion Induction}

We now establish our main theorem.  For the remainder we will consider distance graphs $\mathcal{G}$ with vertices labeled as $V = \{v_1, \ldots, v_n\}$ and when $v_i$ is adjacent to $v_j$ we label their common edge $e_{i,j}$.  Recall that to each such edge we have a distance $\lambda_{e_{i,j}} \in \mathbf{F}_q^*$.  As in the simple example in the previous section, we will apply \cref{distancesTheorem} to delete a single distance constraint, an edge length, and consider the resulting subgraph.  We find it convenient to introduce the shorthand for $x \in \mathbf{F}_q^d$ of
\[
	\sigma_{i,j}(x) := \begin{cases} \sigma_{\lambda_{e_{i,j}}}(x) & \text{if } v_i \text{ and } v_j \text{ are adjacent} \\ 1 & \text{otherwise} \end{cases}
\]
to compactly for $A \subseteq \mathbf{F}_q^d$ the count of
\begin{equation}\label{edCount}
	\mathcal{N}_{\mathcal{G}}(A) := \E_{x_1, \ldots, x_n} \prod_{i = 1}^n 1_A(x_i) \prod_{j = i + 1}^{n} \sigma_{i,j}(x_i - x_j).
\end{equation}
This is \emph{almost} a normalized count for the number of (ordered) isometric copies of $\mathcal{G}$ appearing within $A$.  The only difference is the additional contribution from degenerate configurations that occur when some of the non-adjacent $x_i$ coincide.  We will account for this contribution at the end of the section.  Our main asymptotic here is:

\begin{theorem}\label{degreeAsymptotic} Let $\mathcal{G} = (V,E)$ be a distance graph with $n$ vertices, $m$ edges, and maximum degree $t$.  Let $A \subseteq \mathbf{F}_q^d$ with $|A| = \alpha q^d$ and
\begin{equation}\label{edAlphaSize}
	\alpha \geq 4mq^{t - \frac{d + 1}{2}}.
\end{equation}
Then
\begin{equation}\label{edInduction}
	\Big|\mathcal{N}_\mathcal{G}(A) - \alpha^n \Big| \leq 4 m \alpha^{n - 1} q^{t - \frac{d + 1}{2}}.
\end{equation}
\end{theorem}

\begin{proof}

We proceed by induction on the number of edges, $m$.  In the case $m = 1$, note the maximum degree is $t = 1$ and we may as well assume we only have $n = 2$ vertices, since each isolated vertex contributes a factor of $\alpha$ to $\mathcal{N}_\mathcal{G}(A)$.  Then \eqref{edInduction} is actually weaker by a constant factor than the content of \cref{distancesTheorem} applied with $f = g = 1_A$, so we have one base case already established.

We assume that the theorem has been established for graphs with at most $m - 1$ edges.  It is obvious, but crucial for us, that the maximum degree of subgraphs of $\mathcal{G}$ remains at most $t$ so that we can proceed to apply \eqref{edInduction} to the subgraph of $\mathcal{G}$ formed by removing an edge.  For notational convenience, there is no harm in assuming that $v_1$ and $v_2$ are adjacent.  Set
\[
	f(x_1) := 1_A(x_1) \prod_{j = 3}^{n} \sigma_{1,j} (x_1 - x_j),
\]
and, for $2 \leq i \leq n$,
\[
	g_i(x_i) := 1_A(x_i) \prod_{j = i + 1}^{n} \sigma_{i,j}(x_i - x_j).
\]
Note that $f$ and the $g_i$ together account for all edges of $\mathcal{G}$ except for $e_{1,2}$, which we will be deleting.  With this notation, we can rearrange our count \eqref{edCount} as
\begin{equation}\label{rearrangedCount}
	\mathcal{N}_{\mathcal{G}}(A) = \E_{x_3, \ldots, x_n} \prod_{i = 3}^n g_i(x_i) \E_{x_1, x_2} f(x_1) g_2(x_2) \sigma_{1,2}(x_1 - x_2).
\end{equation}
Applying \cref{distancesTheorem} to the inner average,
\begin{equation}\label{applyDistances}
\E_{x_1, x_2} f(x_1) g_2(x_2) \sigma_{1,2}(x_1 - x_2) = \E_{x_1} f(x_1) \E_{x_2} g_2(x_2) + \mathcal{E}
\end{equation}
where
\[
	|\mathcal{E}| \leq 2q^\frac{1 - d}{2} \| f \|_2 \| g_2 \|_2.
\]
Since $f$ and $g_2$ amount to normalized indicator functions, we can essentially replace these $L^2$ norms with simple averages provided we are careful about normalization.  Since they each account for at most $t - 1$ edges, we have
\begin{align*}
	\| f \|_2^2 &= \E_{x_1} f(x_1)^2\\
	&\leq q^{t - 1} \E_{x_1} f(x_1)
\end{align*}
and similarly $\| g_2 \|^2_2 \leq q^{t - 1} \E_{x_2} g(x_2)$.  Hence,
\[
	|\mathcal{E}| \leq 2 q^{t - \frac{d + 1}{2}} \Big(\E_{x_1} f(x_1)\Big)^{1/2}\Big(\E_{x_2} g(x_2)\Big)^{1/2}.
\]
Together with \eqref{rearrangedCount} and \eqref{applyDistances}, we see
\[
	\Big|\mathcal{N}_\mathcal{G}(A) - \E_{x_1, \ldots, x_n} f(x_1) \prod_{i = 2}^{n} g_i(x_i) \Big| \leq 2 q^{t - \frac{d + 1}{2}} \E_{x_3, \ldots, x_n} \prod_{i = 3}^n g_i(x_i) \| f \|_1^{1/2} \| g_2 \|_1^{1/2}.
\]
This essentially says that, up to an error, the count for copies of $\mathcal{G}$ within $A$ is the same as that for the subgraph of $\mathcal{G}$ formed by deleting the edge $e_{1,2}$.  We invoke our induction hypothesis to observe
\[
	\Big|\E_{x_1, \ldots, x_n} f(x_1) \prod_{i = 2}^{n} g_i(x_i) - \alpha^n\Big| \leq 4(m - 1) \alpha^{n - 1} q^{t - \frac{d + 1}{2}},
\]
in which case we have shown
\begin{equation}\label{edAlmost}
	\Big|\mathcal{N}_\mathcal{G}(A) - \alpha^n\Big| \leq 4(m - 1) \alpha^{n - 1} q^{t - \frac{d + 1}{2}} + 2 q^{t - \frac{d + 1}{2}} \E_{x_3, \ldots, x_n} \prod_{i = 3}^n g_i(x_i) \| f \|_1^{1/2} \| g_2 \|_1^{1/2}.
\end{equation}
Now we need only show that we can bring these two error terms in line.  Applying Cauchy-Schwarz,
\[
	\E_{x_3, \ldots, x_n} \prod_{i = 3}^n f_i(x_i) \| f \|_1^{1/2} \| g_2 \|_1^{1/2} \leq \Big( \E_{x_3, \ldots, x_n} \prod_{i = 3}^n g_i(x_i) \Big)^{1/2} \Big(\E_{x_1, \ldots, x_n} f(x_1) \prod_{i = 2}^n g_i(x_i)\Big)^{1/2}.
\]
By our induction hypothesis, and our assumption that $\alpha$ is not too small, we have both
\begin{align*}
	\E_{x_3, \ldots, x_n} \prod_{i = 3}^n g_i(x_i) &\leq 2\alpha^{n - 2}, \text{ and}\\
	\E_{x_1, \ldots, x_n} f(x_1) \prod_{i = 2}^n g_i(x_i) &\leq 2\alpha^n.
\end{align*}
In particular, together with \eqref{edAlmost}, this allows us to conclude \eqref{edInduction} as desired.
\end{proof}

To extract \cref{introMain}, we must argue that at least one of the many configurations counted by $\mathcal{N}_\mathcal{G}(A)$ is genuine, in the sense that it is made up of distinct vertices.  As one might expect, we will in fact show that when $A$ is large enough, many of these configurations are genuine.  Note that \cref{introMain} follows immediately from the following.

\begin{thm}
Let $n,t \in \mathbf{N}$ and let $A \subseteq \mathbf{F}_q^d$ with $|A| = \alpha q^d$ with $\alpha \geq 12n^2 q^{t - \frac{d + 1}{2}}$.  Then $A$ contains at least $\frac{1}{2}|A|^n q^{-m}$ isometric copies of every distance graph with $n$ vertices, $m$ edges, and maximum vertex degree $t$.
\end{thm}

\begin{proof}
Let $\mathcal{G}$ be a distance graph with $n$ vertices, $m$ edges, and maximum vertex degree $t$ with the same labeling scheme as before.  We introduce the restricted average
\[
	\E_{x_1, \ldots, x_n}^* := q^{-nd} \sum_{\substack{x_1, \ldots, x_n \in \mathbf{F}_q^d \\ x_1, \ldots, x_n \text{ distinct}}}
\]
and consider
\[
	\mathcal{N}_\mathcal{G}^*(A) := \E_{x_1, \ldots, x_n}^* \prod_{i = 1}^n 1_A(x_i) \prod_{j = i + 1}^{n} \sigma_{i,j}(x_i - x_j),
\]
which only counts genuine isometric copies of $\mathcal{G}$ appearing within $A$.  If we are considering a configuration within $A$ that is detected by $\mathcal{N}_\mathcal{G}$, but not by $\mathcal{N}^*_\mathcal{G}(A)$, then some vertices coincide.  We will show that the contribution when $x_n$ is equal to one of $x_1, \ldots, x_{n - 1}$, the contribution is negligible.  In particular, we will bound
\[
	\mathcal{E}_n := \E_{x_1, \ldots, x_{n - 1}} \prod_{i = 1}^{n - 1} 1_A(x_i) \prod_{j = i + 1}^{n - 1} \sigma_{i,j}(x_i - x_j) q^{-d} \sum_{x_n \in \{x_1, \ldots, x_{n - 1}\}} 1_A(x_n) \prod_{i = 1}^{n - 1} \sigma_{i,n}(x_i - x_n).
\]
Of course, we can express this as
\[
	\mathcal{E}_n = q^{-d} \E_{x_1, \ldots, x_{n - 1}} \prod_{i = 1}^{n - 1} 1_A(x_i) \prod_{j = i + 1}^{n - 1} \sigma_{i,j}(x_i - x_j) \sum_{j = 1}^{n - 1} \prod_{i = 1}^{n - 1}\sigma_{i,n}(x_i - x_j).
\]
When the innermost sum is nonzero, its contribution is at least bounded by $(n - 1)q^t$, since $v_n$ has degree at most $t$.  As we are assuming $\alpha$ is not too small, we can invoke \cref{degreeAsymptotic} to conclude
\[
	\mathcal{E}_n \leq 2(n - 1)\alpha^{n - 1} q^{t - d}.
\]
Applying a symmetric argument for coincidences involving vertices other than $x_n$, we have shown
\[
	|\mathcal{N}_\mathcal{G}(A) - \mathcal{N}^*_\mathcal{G}(A)| \leq 2n^2 \alpha^{n - 1} q^{t - d}.
\]
Combining this with \cref{degreeAsymptotic} with our size requirement on $\alpha$ yields
\[
	|\mathcal{N}^*_\mathcal{G}(A) - \alpha^n| \leq 6n^2 \alpha^{n - 1} q^{t - \frac{d + 1}{2}}
\]
which implies that $\mathcal{N}^*_\mathcal{G}(A) \geq \frac{1}{2} \alpha^n$.  The result follows by then clearing normalizations.
\end{proof}

\providecommand{\bysame}{\leavevmode\hbox to3em{\hrulefill}\thinspace}
\providecommand{\MR}{\relax\ifhmode\unskip\space\fi MR }
\providecommand{\MRhref}[2]{%
  \href{http://www.ams.org/mathscinet-getitem?mr=#1}{#2}
}
\providecommand{\href}[2]{#2}


\begin{thebibliography}{10}

\bibitem{bennett16-paths}
Michael Bennett, Jeremy Chapman, David Covert, Derrick Hart, Alex Iosevich, and
  Jonathan Pakianathan, \emph{Long paths in the distance graph over large
  subsets of vector spaces over finite fields}, J. Korean Math. Soc.
  \textbf{53} (2016), no.~1, 115--126. \MR{3450941}

\bibitem{bennett17}
Michael Bennett, Derrick Hart, Alex Iosevich, Jonathan Pakianathan, and Misha
  Rudnev, \emph{Group actions and geometric combinatorics in
  {$\mathbf{F}_q^d$}}, Forum Math. \textbf{29} (2017), no.~1, 91--110.
  \MR{3592595}

\bibitem{bennett16-chains}
Michael Bennett, Alexander Iosevich, and Krystal Taylor, \emph{Finite chains
  inside thin subsets of {$\mathbb{R}^d$}}, Anal. PDE \textbf{9} (2016), no.~3,
  597--614. \MR{3518531}

\bibitem{chapman12}
Jeremy Chapman, M.~Burak Erdo{\u g}an, Derrick Hart, Alex Iosevich, and Doowon
  Koh, \emph{Pinned distance sets, {$k$}-simplices, {W}olff's exponent in
  finite fields and sum-product estimates}, Math. Z. \textbf{271} (2012),
  no.~1-2, 63--93. \MR{2917133}

\bibitem{chatz17}
N~Chatzikonstantinou, A~Iosevich, S~Mkrtchyan, and J~Pakianathan,
  \emph{Rigidity, graphs and {H}ausdorff dimension}, arXiv preprint
  arXiv:1708.05919 (2017).

\bibitem{erdogan05}
M.~Burak Erdo{\u{g}}an, \emph{A bilinear {F}ourier extension theorem and
  applications to the distance set problem}, Int. Math. Res. Not. (2005),
  no.~23, 1411--1425. \MR{2152236}

\bibitem{falconer85}
K.~J. Falconer, \emph{On the {H}ausdorff dimensions of distance sets},
  Mathematika \textbf{32} (1985), no.~2, 206--212. \MR{834490}

\bibitem{greenleaf17}
Allan Greenleaf, Alex Iosevich, and Malabika Pramanik, \emph{On necklaces
  inside thin subsets of {$\mathbb{R}^d$}}, Math. Res. Lett. \textbf{24}
  (2017), no.~2, 347--362. \MR{3685274}

\bibitem{guth15}
Larry Guth and Nets~Hawk Katz, \emph{On the {E}rd{\H o}s distinct distances
  problem in the plane}, Ann. of Math. (2) \textbf{181} (2015), no.~1,
  155--190. \MR{3272924}

\bibitem{iosevich07}
A.~Iosevich and M.~Rudnev, \emph{Erd{\H o}s distance problem in vector spaces
  over finite fields}, Trans. Amer. Math. Soc. \textbf{359} (2007), no.~12,
  6127--6142 (electronic). \MR{2336319}

\bibitem{iosevich16}
Alex Iosevich and Bochen Liu, \emph{Equilateral triangles in subsets of
  {$\mathbb{R}^d$} of large hausdorff dimension}, arXiv preprint
  arXiv:1603.01907 (2016).

\bibitem{iwaniec04}
Henryk Iwaniec and Emmanuel Kowalski, \emph{Analytic number theory}, American
  Mathematical Society Colloquium Publications, vol.~53, American Mathematical
  Society, Providence, RI, 2004. \MR{2061214}

\bibitem{keleti18}
Tam{\'a}s Keleti and Pablo Shmerkin, \emph{New bounds on the dimensions of
  planar distance sets}, arXiv preprint arXiv:1801.08745 (2018).

\bibitem{lidl97}
Rudolf Lidl and Harald Niederreiter, \emph{Finite fields}, second ed.,
  Encyclopedia of Mathematics and its Applications, vol.~20, Cambridge
  University Press, Cambridge, 1997, With a foreword by P. M. Cohn.
  \MR{1429394}

\bibitem{orponen17}
Tuomas Orponen, \emph{On the distance sets of {A}hlfors-{D}avid regular sets},
  Adv. Math. \textbf{307} (2017), 1029--1045. \MR{3590535}

\bibitem{parshall17}
Hans Parshall, \emph{Simplices over finite fields}, Proc. Amer. Math. Soc.
  \textbf{145} (2017), no.~6, 2323--2334. \MR{3626492}

\bibitem{shmerkin17}
Pablo Shmerkin, \emph{On the {H}ausdorff dimension of pinned distance sets},
  arXiv preprint arXiv:1706.00131 (2017).

\bibitem{vinh12}
Le~Anh Vinh, \emph{On kaleidoscopic pseudo-randomness of finite {E}uclidean
  graphs}, Discuss. Math. Graph Theory \textbf{32} (2012), no.~2, 279--287.
  \MR{2974048}

\bibitem{weil48}
Andr\'e Weil, \emph{On some exponential sums}, Proc. Nat. Acad. Sci. U. S. A.
  \textbf{34} (1948), 204--207. \MR{0027006}

\bibitem{wolff99}
Thomas Wolff, \emph{Decay of circular means of {F}ourier transforms of
  measures}, Internat. Math. Res. Notices (1999), no.~10, 547--567.
  \MR{1692851}

\end{thebibliography}
\end{document}